\documentclass[14pt]{amsart}
\address{\newline{\normalsize Laboratory of AGHA, Moscow Institute of Physics and Technology, 9 Institutskiy per., Dolgoprudny,
Moscow Region, 141701, Russia}
\newline{\it E-mail address}: karzhemanov.iv@mipt.ru}
\usepackage{amscd,amsthm,amsmath,amssymb}
\usepackage[dvips]{graphicx}
\usepackage[matrix,arrow]{xy}

\makeatletter\@addtoreset{equation}{section}\makeatother
\renewcommand{\theequation}{\thesection.\arabic{equation}}

\renewcommand{\thesubsection}{\bf\thesection.\arabic{equation}}
\makeatletter\@addtoreset{subsection}{equation}\makeatother

\newtheorem{theorem}[equation]{Theorem}
\newtheorem{prop}[equation]{Proposition}
\newtheorem{lemma}[equation]{Lemma}
\newtheorem{cor}[equation]{Corollary}
\newtheorem{conj}[equation]{Conjecture}

\theoremstyle{definition}
\newtheorem{example}[equation]{Example}

\newtheorem{definition}[equation]{Definition}

\theoremstyle{remark}
\newtheorem{remark}[equation]{Remark}

\renewcommand{\theenumi}{\arabic{enumi}}

\renewcommand{\theenumi}{\arabic{enumi})}

\newcommand{\com}{\mathbb{C}}
\newcommand{\cel}{\mathbb{Z}}
\newcommand{\ra}{\mathbb{Q}}
\newcommand{\re}{\mathbb{R}}
\newcommand{\na}{\mathbb{N}}

\newcommand{\aut}{\text{Aut}}
\newcommand{\p}{\mathbb{P}}

\newcommand{\pic}{\text{Pic}}

\newcommand{\fie}{\textbf{k}}

\newcommand{\spe}{\text{Spec}}
\newcommand{\cha}{\text{char}}

\newcommand{\map}{\longrightarrow}
\newcommand{\ve}{\mathcal{E}}

\title{On characterization of toric varieties}

\author{Ilya Karzhemanov}

\thanks{{\it MS 2010 classification}: 14M25, 14E30}


\begin{document}

\begin{abstract}
We study the conjecture due to V.\,V. Shokurov on characterization
of toric varieties. We also consider one generalization of this
conjecture. It is shown that none of the characterizations holds
in dimensions $\ge 3$ (yet some weaker versions of the
conjecture(s) are verified). In addition, we comment on the recent
paper \cite{brown}, claiming a ``\,proof\," of the conjecture, in
the Appendix at the end.
\end{abstract}

\maketitle

\bigskip

\section{Introduction}
\label{section:int}

\refstepcounter{equation}
\subsection{}
\label{subsection:int-1}

There is an abundance of results on characterization of algebraic
varieties with a transitive group action. Some of those we are
familiar with are \cite{mori-harts}, \cite{kollar-kachi},
\cite{wahl}, \cite{fulton-et-al} (for projective spaces and
hyperquadrics) and \cite{hwang-mok}, \cite{wiesh} (for
Grassmannians and other Hermitian symmetric spaces). At the same
time, not very much is known in this respect for varieties with
other, less transitive group actions. Amongst the first that come
into ones mind are toric varieties and, more generally, reductive
varieties (see \cite{al-bri-1}, \cite{al-bri-2} (and also
\cite{bri-1}, \cite{per}) for foundations of the reductive (resp.
spherical) theory). In the same spirit, there is a related
Hirzebruch problem on describing all compactifications of
$\com^n$, the topic studied in numerous papers (see
\cite{pet-shneider}, \cite{furush-1}, \cite{furush-2} and
\cite{prokhor-coin-3} for instance). Postponing the discussion of
all these matters until some other time let us focus on the toric
case.

To begin with, we mention the paper \cite{ked-wie} (which
elaborates on \cite{jac-94}), where an arbitrary smooth complete
toric variety $T$ is characterized by the property that certain
sheaf $\mathcal{R}_T\in\text{Ext}^1(\mathcal{O}_T^{\oplus
h^{1,1}(X,\com)},\Omega^1_T)$ (called \emph{potential}) splits
into a direct sum of line bundles $\mathcal{O}_T(-D_{\alpha})$.
Unfortunately, this is not quite an effective criterion, and we
are up to a ``\,numerical\,'' one.

\refstepcounter{equation}
\subsection{}
\label{subsection:int-2}

Let $X \longrightarrow Z\ni o$ be an algebraic variety
\footnote{~All varieties, if not specified, are assumed to be
normal, projective and defined over $\com$. We will also be using
freely the notions and facts from the minimal model theory (see
\cite{BCHM}, \cite{kol-mor}, \cite{kollar-et-al}).} over a scheme
(germ) $Z\ni o$. Put $n := \dim X$ and let $n\ge 2$ in what
follows.

Consider a $\ra$\,-\,boundary $D := \sum d_i D_i$, where $D_i$ are
(not necessarily prime) effective Weil divisors on $X$, $0\leq
d_i\leq 1$. Through the rest of the paper $X$ and $D$ will be
subject to the following constraints:

\begin{itemize}

\item the pair $(X,D)$ is log canonical (or \emph{lc} for short),

\item the divisor $-(K_X+D)$ is nef,

\item singularities of $X$ are $\ra$\,-\,factorial (or \emph{$X$ is $\ra$\,-\,factorial}).

\end{itemize}
The last condition is actually redundant (and the first one is too
general) for the forthcoming considerations. Indeed, one may
always apply a dlt modification $h: \tilde{X}\map X$ with
$\ra$\,-\,factorial $\tilde{X}$, divisorially log terminal
$(\tilde{X},h_*^{-1}(D))$ and $K_{\tilde{X}}+h_*^{-1}(D) + \Sigma
\equiv h^*(K_X+D)$ for an $h$\,-\,exceptional divisor
$\Sigma$.\footnote{~$\equiv$ stands for the numerical
equivalence.} Replacing the pair $(X,D)$ by
$(\tilde{X},h_*^{-1}(D) + \Sigma)$ does not effect the arguments
of the present paper.

Let $N^1(X)$ be the N\'eron\,-\,Severi group of $X$. One defines
the number $r(X,D)$ as the dimension of the $\ra$\,-\,vector
subspace in $N^1(X)\otimes\ra$ spanned by all the $D_i$
(alternatively, when one drops the $\ra$\,-\,factoriality
assumption, $r(X,D)$ is defined as the dimension of the
$\ra$\,-\,vector subspace spanned by all the $D_i$ modulo
algebraic equivalence).

Note that $r(X,D)\leq\rho(X) =$ the Picard number of $X$. A finer
relation between these gadgets is provided by the following:

\begin{conj}[V.\,V. Shokurov]
\label{theorem:slava} For $(X/Z\ni o,D)$ as above the estimate
$\sum d_i \leq r(X,D) + \dim X$ holds. Moreover, the equality is
achieved iff the pair $(X,D')$ is (formally) toric for some $D'
\ge \llcorner D\lrcorner$, i.\,e. $X$ is toric and $D'$ is its
boundary (all over $Z\ni o$).
\end{conj}

Clearly, when $X$ is a genuine toric variety and $D$ is its
boundary, with $D_i$ being irreducible and $d_i = 1$ for all $i$,
then $K_X + D \sim 0$ and $\sum d_i = r(X,D) + \dim X$, thus
motivating the last statement of Conjecture~\ref{theorem:slava}.
Let us give another

\begin{example}[{cf. \cite[Example 1.3]{pro-1}}]
\label{example:ex-1} Take $Z := X$ for $(X/X \ni o, D)$ being a
singularity germ. Let $X := (xy + zt = 0)\subset\com^4$ and $D :=
(xy = 0) \cap X$. Then we have $r(X,D) = 1$ and $K_X = 0$. Hence
$\sum d_i = r(X,D) + \dim X$ and $X \ni o$ is obviously toric. In
fact this is not a coincidence as the local version (i.\,e. with
$Z = X$) of Conjecture~\ref{theorem:slava} has been proved in
\cite[18.22\,--\,18.23]{kollar-et-al}. Namely, given $(X/X \ni o,
D)$ such that each $D_i$ is $\ra$\,-\,Cartier, the estimate $\sum
d_i \le \dim X$ holds. Moreover, in the case of equality one has
$X \simeq (\com^n\ni 0)/\frak{A}$, where $\frak{A}$ is a finite
abelian group acting diagonally on $\com^n$ and $D_i$ correspond
to $\frak{A}$\,-\,invariant hyperplanes $(x_i=0)\subset\com^n$
under the factorization morphism $\com^n\map X$. This implies that
the pair $(X,\llcorner D\lrcorner)$ is toric.
\end{example}

\begin{remark}
\label{remark:formal-rel} In view of Example~\ref{example:ex-1},
it is tempting to ask whether $X$ (and $D$) in
Conjecture~\ref{theorem:slava} is by any means related to a toric
variety --- say, whether $(X,\llcorner D\lrcorner)$ is
\emph{formally} (not necessarily regularly or even analytically)
isomorphic to a toric pair? A $\cha\,p > 0$ version of
Conjecture~\ref{theorem:slava} might also be of some interest: for
instance, when the ground field is $\overline{\mathbb{F}}_p$, is
$(X,\llcorner D\lrcorner)$ a toric pair (up to the Frobenius
twist)? Finally, one may consider a weaker version of
Conjecture~\ref{theorem:slava}, with ``$X$ is toric" replaced by
``\,$X$ admits a torification\," (compare with e.\,g.
\cite{lopez-pena-lorscheid}).
\end{remark}

\refstepcounter{equation}
\subsection{}
\label{subsection:int-3}

On trying to probe Conjecture~\ref{theorem:slava} one may assume
that $Z \ne X$ (cf. Example~\ref{example:ex-1}). Then replacing
$Z$ by a formal neighborhood of $o$ we are led to the case of $Z =
\spe\,\com[[t]]$. Taking an embedding
$\com[[t]]\hookrightarrow\com$ one may assume that actually $Z =
o$.

The first proof of Conjecture~\ref{theorem:slava}, for $\dim X =
2$, has appeared in \cite[Theorem 6.4]{slava-com}. The case when
$\dim X = 3$, the pair $(X,D)$ is plt and $K_X+D \equiv 0$ was
treated in \cite{pro-2}, and Conjecture~\ref{theorem:slava} has
been proved in full there under the stated
conditions.\footnote{~Note that under the extremal condition $\sum
d_i = r(X,D) + \dim X$ all (potentially toric) threefolds in
\cite{pro-2} turned out to be actually toric and \emph{smooth}
(see \cite[Theorem 1.2]{pro-2}).} Let us also point out a
birational (``\,rough\,'') version of
Conjecture~\ref{theorem:slava} (cf. Theorem~\ref{theorem:main}
below) proved in \cite{pro-1} assuming the \emph{Weak Adjunction
Conjecture}. Finally, a general strategy towards the proof of
Conjecture~\ref{theorem:slava}, and more, was developed in
\cite{james} and illustrated there (in passing) at some crucial
points. We now recall some of the matters from \cite{james}.

First of all one generalizes Conjecture~\ref{theorem:slava} as
follows. Put
$$
c(X,D) := r(X,D) + \dim X - \sum d_i
$$
(the quantity $c(X,D)$ was called in \cite{james} the
\emph{complexity} of $(X,D)$). Similarly, one defines
$$
ac(X,D) := \rho(X) + \dim X - \sum d_i
$$
(the \emph{absolute complexity}), and let us also introduce
\begin{eqnarray}
\nonumber c(X) := \inf\{c(X,D)\ \vert \ D \ \text{is a
boundary on} \ X \ \text{such that}\\
\nonumber \text{the divisor}\ -(K_X+D) \ \text{is nef and the
pair}\ (X,D)\ \text{is lc}\}
\end{eqnarray}
for consistency, and analogously $ac(X)$ with $c(X,D)$ replaced by
$ac(X,D)$. Then one observes that for $d_i$ all integer the
condition $c(X,D) = 0$ (so that the pair $(X,D)$ is toric
according to Conjecture~\ref{theorem:slava}) is equivalent to
$c(X,D)<1$. This led the author of \cite{james} to make his

\begin{conj}[J. McKernan]
\label{theorem:james} For $(X,D)$ as above the following holds:
\begin{enumerate}
\item\label{cl-1} $c(X)\ge 0$,
\smallskip
\item\label{cl-2} if $ac(X,D)<2$, then $X$ is a rational variety,
\smallskip
\item\label{cl-3} if $c(X,D)<1$, then there is a divisor $D'$ such that the
pair $(X,D')$ is toric. Moreover, $\llcorner D\lrcorner\subseteq
D'$ and $D'-S$ is linearly equivalent to a divisor with support in
$D$, where $S$ is either empty or an irreducible divisor.
\end{enumerate}
\end{conj}

One obviously has the implication Conjecture~\ref{theorem:james}
$\Longrightarrow$ Conjecture~\ref{theorem:slava}.

\begin{remark}
\label{remark:other-path} Note that it is not possible to loose
any of the assumptions in Conjectures~\ref{theorem:slava} and
\ref{theorem:james}. Indeed, for $X:=\p^1\times\p^1\times E$ and
$D := 0\times\p^1\times E+\infty\times\p^1\times E+\p^1\times
0\times E+\p^1\times \infty\times E$, $E$ is an elliptic curve, we
have $c(X,D)=c(X)=1$, but $X$ is not even rationally connected
(cf. Conjecture~\ref{theorem:james},\,\ref{cl-2}). Also, taking
$X:=\mathbb{F}_m$ and $D:=2E_{\infty} +
\displaystyle\sum_{i=1}^{m+2}F_i$, where $E_{\infty}$ is the
negative section and $F_i$ are fibers of the natural projection
$\mathbb{F}_m\map\p^1$ (so that $K_X+D = 0$), we get $c(X,D)\le
0$, but the pair $(X,D)$ is non\,-\,toric (cf.
Conjecture~\ref{theorem:james},\,\ref{cl-3}). At the same time, as
one immediately verifies, slightly more general versions of
Conjectures~\ref{theorem:slava} and \ref{theorem:james} studied in
\cite{slava-com}, \cite{pro-1}, \cite{pro-2} and \cite{james} are
equivalent to those stated above.
\end{remark}

\refstepcounter{equation}
\subsection{}
\label{subsection:int-5}

The aim of the present paper is to show that
Conjecture~\ref{theorem:slava} is not true as stated when $\dim
X\ge 3$ (cf. Remark~\ref{remark:no-tor-sl}). But before spelling
out the details let us give a (one of many possible)

\begin{definition}
\label{theorem:f-toric} We call an algebraic variety $X$
\emph{fake toric} (or \emph{$\frak{f}$\,-\,toric} for short) if
there exists a $\ra$\,-\,boundary $D$ on $X$, a
\emph{quasi\,-\,toric boundary} (or \emph{qt\,-\,boundary}), such
that the pair $(X,D)$ matches all the conditions in
{\ref{subsection:int-2}} (up to passing to a
$\ra$\,-\,factorialization as usual), the pair $(X,\llcorner
D\lrcorner)$ is non\,-\,toric, but $c(X,D) < 1$. Denote by
$\frak{T}^{\frak{f},n}$ the class of all $\frak{f}$\,-\,toric
varieties of dimension $n$.
\end{definition}

Here is our main result:

\begin{theorem}
\label{theorem:main-1} The class
$\frak{T}^{\frak{f},n}\ne\emptyset$ for every $n\ge 3$. More
precisely, there exists $X\in\frak{T}^{\frak{f},n}$ with a
qt\,-\,boundary $D$ such that $K_X+D\equiv 0$, $ac(X,D) =
\displaystyle\frac{3}{4}$.
\end{theorem}

Theorem~\ref{theorem:main-1} is proved in
Section~\ref{section:exs-1} below (after some simple variants of
Conjecture~\ref{theorem:james} --- as discussed in
{\ref{subsection:int-6}}). We stress that our construction of
varieties from $\frak{T}^{\frak{f},n}$ was very much motivated by
the exposition in \cite{james}. In fact, the arguments of
\cite{james} rely on essentially two assertions, namely that an
algebraic variety isomorphic to a toric variety in codimension $1$
is also toric (see \cite[\S 3]{james}), and that the divisor
$K_X+D$ has an \emph{integral} $\ra$\,-\,complement (cf.
\cite{pro-comp}, \cite{slava-com}), i.\,e. one can always find an
integral Weil divisor $D'\ge D$ (for $D$ considered up to the
linear equivalence) such that the pair $(X,D')$ is lc and
$K_X+D'\equiv 0$ (see \cite[\S 4]{james}).
Conjecture~\ref{theorem:james} then follows from these two
assertions and a $\dim X$\,-\,inductive argument.

The outlined strategy applies well when $\dim X = 2$ and justifies
Conjecture~\ref{theorem:slava} in this case (cf. \cite{slava-com},
\cite[Proposition 2.1]{pro-2}, \cite[Theorem 8.5.1]{pro-comp}). At
the same time, rational scrolls (a.\,k.\,a. $\p^N$\,-\,bundles
over $\p^1$) are all toric, so that it is reasonable to test
Conjecture~\ref{theorem:james} on $X$ related to a
projectivization $\p(\ve)$ for some indecomposable vector bundle
$\ve$ on $\p^2$. And it is this $X$ for which the method of
\cite{james} together with Conjecture~\ref{theorem:james} break.
More specifically, when $\dim X = 3$ one obtains $X$ as the blowup
of a quadric $Q\subset\p^4$ in a line $\ell\subset Q$, followed by
factorizing by two commuting $\linebreak\cel\slash 2$\,-\,actions.
Both $\cel\slash 2$ act on $Q$ preserving $\ell$ and are
constructed as follows. One $\cel\slash 2$ corresponds to the
Galois action for the $2:1$ projection $Q\map\p^3$ and another
$\cel\slash 2$ is a lift to $Q$ of the Galois involution on $\p^3$
corresponding to the quotient morphism $\p^3\map\p(1,1,1,2)$ (one
can easily see such a lifting does exist). It is then an exercise
to find $D$ as in Theorem~\ref{theorem:main-1} and simple fan
considerations show that $X$ is non\,-\,toric (see
Section~\ref{section:exs-1} for further details). Finally, we
observe in this way that so obtained $X$ is \emph{singular}, while
for smooth $X$ Conjecture~\ref{theorem:slava} has been proved
recently in \cite{yao}.

\begin{remark}
\label{remark:com-on-mista} Let us describe, for consistency, the
above mentioned ``\,toric\,-\,in\,-\,codimension\,-\,$1$\,'' and
``\,$\ra$\,-\,complementary\,'' assertions a bit more thoroughly.
Firstly, in the former case given two varieties $X_1,X_2$ (we
assume both $X_i$ to be $\ra$\,-\,factorial), with $X_1$ toric and
$\theta: X_1\dashrightarrow X_2$ an isomorphism in codimension
$1$, it is easy to see that the indeterminacy locus of $\theta$ is
torus\,-\,invariant (use \cite[Lemma 6.39]{kol-mor} for instance).
Furthermore, as soon as $X_1$ is a Mori dream space (see
\cite{hu-keel}), $\theta$ can be factored into a sequence of
torus\,-\,invariant $\Delta$\,-\,flips with respect to a movable
divisor $\Delta$ on $X_1$, so that $\theta$ and $X_2$ are also
toric. This may work for instance when both $X_i$ admit integral
boundaries $D_i$ and birational contractions $f_i:X_i\map Y$ such
that $c(X_i,D_i)=0$ and $Y$ (hence also $(Y,f_{i*}(D_i))$) is
toric (see the discussion in \cite{james} right
before$\slash$after Definition\,-\,Lemma 3.4).\footnote{~Note
however that there is a substantial gap in \cite{james} at this
point. Namely, let $E_i$ be the $f_i$\,-\,exceptional divisor,
inducing a discrete valuation $v_i$ on the field $\com(Y)$. Then
it is claimed in \cite{james} that (for $(X_i,D_i)$, etc. as
given, but without any assumption on $c(X_i,D_i)$) there is a
sequence of toric blowups of $Y$ extracting $v_i$ (starting with
the blowup of $f_i(E_i)$). But this does not occur in general
(globally at least) because the scheme $f_i(E_i)$ may not be
reduced (this is a popular spot of erroneous usage of \cite[Lemma
2.45]{kol-mor} --- replacing the initial \emph{scheme} $f_i(E_i)$
by $f_i(E_i)_{red}$). For example, consider $Y := \com^2$ with
(toric) boundary $\Delta := (xy = 0)$, and let $f_1: X_1\map Y$ be
the blowup of the scheme $Z:=((x+y^2)^2=y^3=0)$ (supported at
$(0,0)$). In other words, $f_1$ is the weighted blowup with
weights $(3,2)$, so that $K_X+f_{1*}^{-1}\Delta =
f_1^*(K_Y+\Delta) - E_1$ and the pair
$(X_1,D_1:=f_{1*}^{-1}\Delta+E_1)$ is lc. However, the map $f_1$
can not be extended to a toric morphism
$\widetilde{f_1}:\widetilde{X_1}\map\p^2$ for toric surface
$\widetilde{X_1}$ compactifying $X_1$ (resp. $\p^2$ compactifying
$Y$) because $((x+y^2)^2,y^3)$, the defining ideal of $Z$, does
not coincide with $(x^2,y^3)$. One can find non\,-\,toric
$\widetilde{X_1}$ and $\widetilde{f_1}$ though (compare with
$(X,D)$ in the proof of Lemma~\ref{theorem:ac-for-x-0} below).}
But for general $\ra$\,-\,boundaries $D_i$, $f_{i*}(D_i)$ need not
be supported on the toric boundary of $Y$, which brings us to the
``$\ra$\,-\,complementary" part. In the latter case however, one
may expect the \emph{exceptional} complements do not occur for the
$\frak{f}$\,-\,toric pairs, as well as for the toric ones (cf.
\cite{slava-com}, \cite{pr-compl-ex}, \cite{mar-pr-1},
\cite{mar-pr-2}, \cite{ishii-pro}), and so the boundaries
$D_i,f_{i*}(D_i)$ can probably be made toric.
\end{remark}

\refstepcounter{equation}
\subsection{}
\label{subsection:int-6}

We conclude by stating some positive versions of
Conjecture~\ref{theorem:james} we were able to find:

\begin{prop}[Yu.\,G. Prokhorov]
\label{theorem:prox-comp} Let $(X,D)$ be as in
{\ref{subsection:int-2}}, $K_X + D \equiv 0$, $d_i = 1$ for all
$i$. Then

\begin{enumerate}

\item\label{it-A}

$c(X,D)\ge 0$,

\smallskip

\item\label{it-B} if $c(X,D)=0$, then $X$ is rationally connected.
Furthermore, if $X$ has only terminal singularities, then $D$ has
a rationally connected component.

\end{enumerate}

\end{prop}

(Proposition~\ref{theorem:prox-comp} is proved in
Section~\ref{section:aux}. The arguments are entirely due to
Yu.\,G. Prokhorov.)

\begin{theorem}
\label{theorem:main} For $(X,D)$ as in
Proposition~\ref{theorem:prox-comp}, if $ac(X,D)=0$, then $X$ is
rational. Moreover, if $ac(X,D) = 1$ and $X$ is rationally
connected, then $X$ is rational as well.
\end{theorem}

\begin{remark}
\label{remark:ac-vs-a} In Theorem~\ref{theorem:main}, it might be
possible to replace $ac(X,D)$ (resp. $D$ being integral) by
$c(X,D)$ (resp. $\llcorner D\lrcorner\ne 0$), but we could not
extend the arguments to this setting.
\end{remark}

It is in the proof of Theorem~\ref{theorem:main} where we use the
nice inductive trick suggested in \cite{james}. Namely, since
$\text{Supp}\,D$ consists of $\ge\rho(X)+1$ (irreducible)
divisors, one can find a map (a ``\,Morse function\,") $f:
X\map\p^1$ whose fibers satisfy the hypotheses of
Theorem~\ref{theorem:main}, after what we argue by induction on
$\dim X$ (see Section~\ref{section:pro} for further details).

\bigskip

\thanks{{\bf Acknowledgments.} I am grateful to Yu.\,G. Prokhorov and D. Stepanov for introducing me to the subject treated in this
paper and for helpful conversations. Also the work owes much to
hospitality of the Courant Institute and Max Planck Institute for
Mathematics. Finally, partial financial support was provided by
CRM fellowship, NSERC grant, World Premier International Research
Initiative (WPI), MEXT, Japan, and by Grant-in-Aid for Scientific
Research (26887009) from Japan Mathematical Society (Kakenhi).

\bigskip

\section{Proof of Proposition~\ref{theorem:prox-comp}}
\label{section:aux}

\refstepcounter{equation}
\subsection{}
\label{subsection:pre-2}

Firstly, as $K_X\equiv -D\ne 0$, we can run the $K_X$\,-\,MMP (see
\cite[Corollary 1.3.2]{BCHM}). Furthermore, since $K_X+D\equiv 0$,
the components of $D$ are not contracted on each step (see e.\,g.
\cite[Lemma 3.38]{kol-mor}). Note also that the quantity $c(X,D)$
does not increase on each step. Hence we may assume there is an
extremal contraction $\phi: X \map Y$ which is a Mori fiber space.
Finally, as one can easily see by restricting $D$ to a general
fiber $F$ of $\phi$, there is a horizontal component $D_0 \subset
D$, i.\,e. $\phi(D_0) = Y$. One may take $D_0$ to be irreducible.

\begin{lemma}[{cf. \cite[Corollary 3.7.1]{pro-2}}]
\label{theorem:prox-comp-l} We have either $D_0\cap
D_i\ne\emptyset$ for all $i$ or there exists $j$ such that
$D_0\cap D_j = \emptyset$. In the latter case, $\phi: X
\longrightarrow Y$ is a conic bundle with two (generic) sections
$D_0,D_j$.\footnote{~Both $D_0$ and $D_j$ may contain some fiber
components.}
\end{lemma}

\begin{proof}
Suppose first that $\dim X/Y > 1$. Then, since $D_0\big\vert_F$ is
ample, one gets $D_0\cap D_i\ne\emptyset$ for all $i$. Further, if
$\dim X/Y = 1$ and $D_0$ does not intersect at least two of
$D_i$, say $D_0 \cap D_1 = D_0 \cap D_2 = \emptyset$, then $D_0,
D_1, D_2$ are multisections of $\phi$ and we obtain
$$
0 \equiv \big(K_X + D\big)\big\vert_F = K_F + D\big\vert_F > 0
$$
on $F = \p^1$, a contradiction. Thus there is $D_j$ with $D_j\cap
D_0 = \emptyset$. Then, since $K_F = \mathcal{O}_{\p^1}(-2)$, we
have $D_0\cdot F = D_j\cdot F = 1$. Hence both $D_0,D_j$ are
sections of $\phi$.
\end{proof}

Consider the divisor
$$
B := \sum_{i\neq 0,\ D_i\cap D_0\neq \emptyset} D_i
$$
and the normalization $\nu: D_0^{\nu} \longrightarrow D_0$. Let
$B^{\nu} := \text{Diff}_{\scriptscriptstyle D_0}(B)$ be the
different of $B$ on $D_0^{\nu}$ (see e.\,g. \cite[\S
3]{sho-perestrojki}). Note that
$$
K_{D_0^{\nu}} + B^{\nu} = \nu^*\big(\big(K_X +
D\big)\big\vert_{D_0}\big)\equiv 0
$$
in our case.

Further, cutting with hyperplane sections we obtain that $B^\nu =
N+\nu^*B$, where $N$ is the preimage of the non\,-\,normal locus
on $D_0$ (cf. \cite[\S 3]{sho-perestrojki}). Also, by Inversion of
Adjunction (see \cite{kawakita}), the pair $(D_0^{\nu},B^{\nu})$
is lc (because $(X,D)$ is this). Hence we may replace the pair
$(X,D)$ with $(D_0^{\nu},B^{\nu})$ (cf. {\ref{subsection:int-2}}).
Moreover, we have
$$
B^{\nu} = N+\nu^*B = N+\sum_{i\neq 0} \nu^*D_i,
$$
which easily yields $c(D_0^{\nu},B^{\nu})\le c(X,D)$ (see
Lemma~\ref{theorem:prox-comp-l}).
Proposition~\ref{theorem:prox-comp},\,\ref{it-A} follows by
induction on the dimension (note that $B^{\nu}$ is integral by the
construction).

Suppose now that $c(X,D)=0$. Then $c(D_0^{\nu},B^{\nu})=0$ and
$D_0$ is rationally connected by induction. Furthermore, since the
fiber $F$ is a Fano variety with only log terminal singularities,
it is rationally connected (see \cite[Corollary
1.3]{chris-james}). This implies that $X$ is rationally connected
as well (see \cite[Corollary 1.3]{gr-har-st}). Finally, if $X$ has
only terminal singularities, then its dlt modification $\tilde{X}
\longrightarrow X$ is a small birational morphism, and
Proposition~\ref{theorem:prox-comp},\,\ref{it-B} follows.

\begin{cor}
\label{theorem:con-lemma} For $X$ as in
Proposition~\ref{theorem:prox-comp} (in fact for any rationally
connected $X$) and any $\ra$\,-\,Cartier divisors
$L_1,L_2\in\mathrm{Pic}(X)\otimes\ra$, we have $L_1\equiv L_2 \iff
L_1\sim_{\ra}L_2$.
\end{cor}

\begin{proof}
One may replace $X$ by its resolution $X'$. Let $g: X' \map X$ be
a birational contraction. Let us also replace each $L_i$ by
$g^*(L_i)$. The assertion now follows from $h^1(\mathcal{O}_{X'})
= 0$ ($= \dim\text{Alb}_{X'}+1$) provided by
Proposition~\ref{theorem:prox-comp},\,\ref{it-B}. Indeed, for
$L_1,L_2$ both Cartier and effective the condition $L_1\sim L_2$
(equivalent to $L_1\equiv L_2$ due to $\text{Alb}_{X'}=0$) simply
asserts that there is a rational map (a.\,k.\,a. a function in
$\com(X)$) $f:X\dashrightarrow\p^1$ with
$f^{-1}(0)=L_1,f^{-1}(\infty) = L_2$. The argument for arbitrary
$\ra$\,-\,Cartier $L_i$ is similar.
\end{proof}

\bigskip

\section{Proof of Theorem~\ref{theorem:main}}
\label{section:pro}

\refstepcounter{equation}
\subsection{}
\label{subsection:pro-1}

We will assume for what follows that $ac(X,D) = c(X,D)\le 1$. The
proof of Theorem~\ref{theorem:main} will go essentially by
induction on both $\dim X$ and $\rho(X)$. In the former case,
$\dim X = 2$ is the base of induction (cf. the beginning of
{\ref{subsection:int-3}}), while in the latter case we have the
following:

\begin{prop}
\label{theorem:rat-1} The assertion of Theorem~\ref{theorem:main}
holds when $\rho(X) = 1$.
\end{prop}

\begin{proof}
After taking the cyclic coverings of $X$ w.\,r.\,t. various $D_i$
(see e.\,g. \cite[\S 2]{sho-perestrojki}), the arguments in the
proof of \cite[Corollary 2.8]{pro-1} apply and show that
$$
X \simeq \p^n\slash\frak{A}\qquad\text{or}\qquad Q\slash\frak{B},
$$
where $Q\subset\p^{n+1}$ is a smooth quadric,
$\frak{A}\subset\aut\,\p^n$ (resp. $\frak{B}\subset\aut\,Q$) is a
finite abelian group with linearized action on $\p^n$ (resp.
$\p^{n+1}$), and the hyperplane sections $D_i$ correspond to the
only eigen vectors in $H^0(\p^n,\mathcal{O}_{\p^n}(1))$ (resp. in
$H^0(Q,\mathcal{O}_Q(1))$) which $\frak{A}$ (resp. $\frak{B}$)
scales non\,-\,trivially on. Note that the case $X =
\p^n\slash\frak{A}$ has been treated already in \cite[Corollary
2.8]{pro-1} and in fact the pair $(X,D)$ turns out to be toric.

Suppose now that $X = Q\slash\frak{B}$. Let $Q^{\frak{B}}$ be the
locus of $\frak{B}$-fixed points on $Q$.

\begin{lemma}
\label{theorem:rat-1-l} $Q^{\frak{B}}\ne\emptyset$.
\end{lemma}

\begin{proof}
Firstly, one easily finds a $\frak{B}$\,-\,invariant line
$l\subset\p^{n+1}$, so that the locus $l\cap Q$ is also
$\frak{B}$\,-\,invariant. Now, if $l \cap Q$ is a finite set, then
$|l \cap Q|\le 2$. Put $l \cap Q = \{P_1,P_2\}$ for some $P_i\in
Q$. Then it is evident that $P_i\in Q^{\frak{B}}$ (after possibly
replacing $l$ by another $\frak{B}$\,-\,invariant line). Finally,
the case when $|l \cap Q| = \infty$, i.\,e. $l\subset Q$, is
obvious.
\end{proof}

Pick any $P \in Q^{\frak{B}}$ and consider the
($\frak{B}$\,-\,equivariant birational) linear projection $Q
\dashrightarrow \p^n$ from $P$. Then the $\frak{B}$\,-\,action
descends to that on $\p^n$ and both $Q/\frak{B},\p^n/\frak{B}$ are
birationally isomorphic, with rational $\p^n/\frak{B}$.
Proposition~\ref{theorem:rat-1} is proved.
\end{proof}

\begin{remark}
\label{remark:commen-on-prop-rat} In general, when $X$ and each of
$D_i$ are defined over a field $\fie$ ($\cha\,\fie = 0$), one can
easily adjust the arguments from the proof of
Proposition~\ref{theorem:rat-1} to show that $X$ is
$\fie$\,-\,rational, provided
$\rho\big(X\otimes_{\fie}\bar{\fie}\big) = 1$ over the algebraic
closure $\bar{\fie}$.
\end{remark}

\refstepcounter{equation}
\subsection{}
\label{subsection:pro-2}

Let $\phi: X \map Y$ be an extremal contraction. Put $D_Y :=
\phi_*(D),D_{Y,i} := \phi_*(D_i)$, and suppose that $\phi$ is
\emph{divisorial}. Note that $\phi_*K_X = K_Y$ because the
singularities of $Y$ are all rational. Hence $K_Y + D_Y =
\phi_*(K_X + D)\equiv 0$ and $c(Y,D_Y)$ is defined.

\begin{lemma}
\label{theorem:alb-is-a-morphism} Given $Y,\phi$ as above, the
assertion of Theorem~\ref{theorem:main} holds for $X$.
\end{lemma}

\begin{proof}
It is clear that $0\le c(Y,D_Y)\le c(X,D)$. Then the claim follows
from Proposition~\ref{theorem:rat-1} and induction on $\rho(X)$.
\end{proof}

We now turn to the construction of a pencil on $X$ as was
indicated in {\ref{subsection:int-6}}. Firstly, we may assume that
$D_1\sim_{\ra}\sum \delta_i D_i=:D'_1$ for some $\delta_i\in\ra$,
where $\delta_2\ne 0$, say (see
Corollary~\ref{theorem:con-lemma}). Secondly, we may assume
without loss of generality that all $\delta_i\ge 0$ (replacing, if
necessary, $D_1$ by $D_1+\sum_j\delta'_jD_j$ for some
$\delta_j\in\ra_{\ge 0}$), which gives a rational map $f:
X\dashrightarrow \p^1$ with $kD_1 = f^{-1}(0)$ and $kD'_1 =
f^{-1}(\infty)$ for some integer $k\ge 1$. Finally, since the
indeterminacies of $f$ are located on some of $D_i$, we will
assume (for the sake of simplicity) that $f$ is undefined
precisely on $D_1\cap D_2$ (in general, one just has to consider a
larger number of $D_j$, satisfying $D_1\cap\bigcup D_j =$ the
indeterminacy locus of $f$).

Further, there is a local analytic isomorphism
\begin{equation}
\nonumber \big(X,D_1+D_2,D_1\cap D_2\big) \simeq
\big(\com^n,\{x_1x_2 =
0\},\{x_1=x_2=0\}\big)\slash\cel_m(1,q,0,\ldots,0)
\end{equation}
at the general point of $D_1\cap D_2$, where $m,q\in\na,(m,q)=1$
(see e.\,g. \cite[Proposition 3.9]{sho-perestrojki}). In
particular, if $\phi:\tilde{X}\map X$ is the blowup of $D_1\cap
D_2$, then we have $K_{\tilde{X}} + \phi_*^{-1}(D) + E \equiv
\phi^*(K_X+D)$ for the $\phi$\,-\,exceptional divisor $E$. Thus
after taking a number of subsequent blowups of $D_1\cap D_i$ we
may assume that $f: X\map \p^1$ is \emph{regular}.

\refstepcounter{equation}
\subsection{}
\label{subsection:pro-4}

Let $F$ be a fiber of $f$. Note that $K_X\cdot Z = -D\cdot Z \le
0$ for generic $F$ and a curve $Z\subset F$.

\begin{lemma}
\label{theorem:ex-di} $D\cdot Z>0$ for some $Z$ and $F$ as above.
\end{lemma}

\begin{proof}
Assume the contrary. Then $D_i\cdot Z \le 0$ for all $i$ and
$Z\subset F$. Hence $f(D_i) = \text{pt}$ for all $i$. But the
latter is impossible because $D_i$ generate $\pic(X)\otimes\ra$.
\end{proof}

Among those $Z\subset F$ with $K_X\cdot Z<0$ (see
Lemma~\ref{theorem:ex-di}) there is an extremal ray $R$ of the
cone $\overline{NE}(X)$. The curve representing $R$ sits in some
irreducible component $F'\subseteq F$, so that once $F$ is
reducible, $F\setminus F'\ne \emptyset$, we get an extremal
birational contraction $\phi: X\map Y$ over $\p^1$. We eliminate
the case when $\phi$ is flipping by applying the $K_X$\,-\,flip
and termination of MMP with scaling. On the other hand, if $\phi$
is divisorial (in which case it contracts $F'$), then we are done
by Lemma~\ref{theorem:alb-is-a-morphism}.

Thus we may assume that $\rho(X) = 2$ (i.\,e. $f: X\map \p^1$ is a
Mori fibration).

\begin{lemma}
\label{theorem:ex-div-con} $F$ is $\ra$\,-\,factorial,
$\rho(F)=1$, $D'_1=D_2$ and $D_1,D_2$ are the only components of
$D$ contracted by $f$.
\end{lemma}

\begin{proof}
This is proved in \cite[Proposition 3.6, Lemma 2.10]{pro-1}.
\end{proof}

Now exactly as in \cite[Lemma 2.10,\,(iii)]{pro-1},
Lemma~\ref{theorem:ex-div-con},
Remark~\ref{remark:commen-on-prop-rat} and induction on $\dim$
imply rationality of $X$. Theorem~\ref{theorem:main} is completely
proved.

\bigskip

\section{Proof of Theorem~\ref{theorem:main-1}}
\label{section:exs-1}

\refstepcounter{equation}
\subsection{}
\label{subsection:pr-11-1}

Let us introduce the following two automorphisms of $\p^4$:
$$
g_1: [x:y:z:t:w] \mapsto [y:x:z:t:w]
$$
and
$$
g_2: [x:y:z:t:w] \mapsto [-x:-y:-z:t+w:-w].
$$

Consider the group $G$ generated by $g_1,g_2$. We have $G =
\cel\slash 2 \times \cel\slash 2$ (although the $G$\,-\,action is
\emph{not linear}). Let also
$$
\ell := (x - y = t = w = 0) \subset \mathbb{P}^4
$$
be the $G$\,-\,invariant line. Note that $\ell \subset Q$ for the
quadric $Q \subset \mathbb{P}^4$ given by the equation
\begin{equation}
\label{eq-Q-iii} (x + y - w)w + (x - y)^2 + (t + w)t + zw = 0.
\end{equation}
This $Q$ is obviously $G$\,-\,invariant and has $o :=
[1:1:-2:0:0]$ as its unique singular point. Note also that $g_1$
(resp. $g_2$) acts \emph{identically} on the hyperplane $H_1 := (x
- y = 0)$ (resp. $H_2 := (t + w/2 = 0)$).

\refstepcounter{equation}
\subsection{}
\label{subsection:pr-11-2}

Let $\phi: W \map Q$ be the blowup of $\ell$. Then $\phi$ is
$G$\,-\,equivariant by construction. Set $X := W \slash G$
together with the quotient map $p: W \map X$. Let us also define
$\Delta := \phi^*\mathcal{O}_Q(1)$, $\Sigma := \phi^{-1}\ell$,
$R_i := \phi_*^{-1}H_i$. It is easy to see that $\Sigma$ consists
of two irreducible surfaces $\Sigma_1$ and $\Sigma_2$ (cf.
Remark~\ref{remark:structure-of-sigma} below).

\begin{lemma}
\label{theorem:ram-div-of-p} $R_1 \cup R_2 \cup \Sigma \subset W$
is precisely the codimension $1$ ramification locus of $p$.
\end{lemma}

\begin{proof}
Threefold $W$ (resp. $Q$) can be identified with the $\text{Proj}$
of the $\com$\,-\,algebra $A_W := \displaystyle\bigoplus_{k \ge 0}
H^0(W,k\Delta)\oplus H^0(W,k\Sigma)$ (resp. $A_Q :=
\displaystyle\bigoplus_{k \ge 0} H^0(W,k\Delta)$) so that the
inclusion $A_Q \subset A_W$ corresponds to $\phi$.

Now, considering the $G$\,-\,invariant subalgebras in $A_W,A_Q$
and taking $\text{Proj}$ we get the following commutative diagram
of $G$\,-\,equivariant morphisms:
\begin{equation}
\label{diag-imp} \xymatrix{
W\ar@{->}[d]_{\phi}\ar@{->}[r]^{p_1}&W\slash g_1\ar@{->}[d]\ar@{->}[r]^{p_2}&X\ar@{->}[d]\\
Q\ar@{->}[r]&Q \slash g_1 =
\mathbb{P}^3\ar@{->}[r]&\mathbb{P}^3\slash g_2}
\end{equation}
where the vertical (resp. horizontal) arrows signify birational
(resp. $2:1$) morphisms, $p = p_2 \circ p_1$ and $G$ acts on
$\mathbb{P}^3 = H_1$ via $g_2$ as follows:
$$
[x:z:t:w] \mapsto [-x:-z:t + w:-w].
$$
In particular, $g_2$ acts identically on the hyperplane $(t + w =
0) \subset \mathbb{P}^3$, which implies that $\mathbb{P}^3\slash
g_2 = \mathbb{P}(1,1,1,2)$. Further, consider the affine chart
$\com^4_{x,y,t,w}$ of the blowup of $\ell \subset \p^4$, where $z
= 1$ and $t,w$ are replaced by $t(x-y),w(x - y)$, respectively.
Then the local equations of $\Sigma$ in $\com^4_{x,y,t,w}$ are
(cf. \eqref{eq-Q-iii})
$$
x - y = w(x + y + 1) = 0.
$$
From this it is immediate that $g_1 = 1$ and $g_2 \ne 1$ on both
$\Sigma_i$.

It follows that the ramification divisor of the quotient morphism
$p_1: W \map W\slash g_1$ (resp. of $p_2$) is $R_1$ (resp.
$p_1(R_2) \cup p_1(\Sigma)$). This exactly means that $R_1 \cup
R_2 \cup \Sigma$ is the codimension $1$ ramification locus of $p$.
\end{proof}

Let $\alpha: X \map \mathbb{P}^3\slash g_2 = \mathbb{P}(1,1,1,2)$
be as in {\eqref{diag-imp}}. This is an extremal birational
contraction of the surface $\widetilde{\Sigma} := p(\Sigma)$
(scheme\,-\,theoretic image). There is another extremal
contraction on $X$ which we now describe (this will be used later
when computing $ac(X)$).

\begin{remark}
\label{remark:structure-of-sigma} After an appropriate coordinate
change one may assume that $Q = (xy - zt = 0)$ and $\ell = (x = y
= z = 0)$. Then a simple local computation shows that $\Sigma =
\mathbb{F}_1 \cup \p^2$, with $\mathbb{F}_1 \cap \p^2 = $ the
$(-1)$\,-\,curve on $\mathbb{F}_1$, and $W$ has at most ordinary
double points. This implies in particular that the Picard number
of a $\mathbb{Q}$\,-\,factorialization of $W$ equals $4$.
Similarly, since $\mathrm{Pic}(X) = \mathrm{Pic}(W)^G$ (the
$G$\,-\,invariant part), the Picard group of $X$ is generated by
$p(\Delta)$ and (reducible) $\widetilde{\Sigma}$, while the class
group of $X$ has rank $3$ because two of the
(non\,-\,$\mathbb{Q}$\,-\,Cartier) irreducible components of the
divisor $\phi_*^{-1}(w = 0)$ are identified under $p$.
\end{remark}

\refstepcounter{equation}
\subsection{}
\label{subsection:pr-11-3}

The blowup $\phi$ resolves indeterminacies of the projection $Q
\dashrightarrow \mathbb{P}^2$ from $\ell$. Then, since $\rho(W) =
2$, it is plain that the induced morphism $\pi: W \map \p^2$ is a
$G$\,-\,equivariant extremal contraction.

Further, we put $L := \Delta - \Sigma$, so that $\pi$ is given by
the linear system $|L|$. It follows from the construction of
$\p^2$ and $g_i$ in {\ref{subsection:pr-11-1}} that $G$ acts
\emph{faithfully} on $\p^2 = \pi(W)$. Then we have $\p^2\slash G
\simeq\p^2$ and the quotient morphism $\p^2 \map \p^2\slash G$ is
given by a linear subsystem $\mathcal{L} \subset |2\pi_*L|$. In
particular, $\mathcal{L}$ consists of certain $G$\,-\,invariant
global sections of $2\pi_*L$, and hence identifying
$\pi^*\mathcal{L}$ with a linear system on $X$ we get a morphism
$\pi_X: X \map \p^2$ (the quotient of $\pi$) which fits into a
commutative diagram
\begin{equation}
\nonumber \xymatrix{
W \ar@{->}[d]_{\pi}\ar@{->}[r]^{p}&X\ar@{->}[d]^{\pi_X}\\
\p^2\ar@{->}[r]&\mathbb{P}^2\slash G = \p^2}
\end{equation}
This $\pi_X$ is the second extremal contraction on $X$ (cf.
Remark~\ref{remark:structure-of-sigma}).

We observe next that the surface $L' := (\phi^*w = 0) - 2\Sigma$
on $W$ is $G$\,-\,invariant and the corresponding $G$\,-\,action
is \emph{faithful} on it (cf.
{\ref{subsection:pr-11-1}}).\footnote{~To simplify the notation we
treat $w$ and $z$ below as global sections of $\mathcal{O}_Q(1)$.}
This implies that $\deg p\big\vert_{L'} = 4$ and thus we get
$p_*L' \equiv 4\widetilde{L'}$ for $\widetilde{L'} := p(L')$.
Similarly, for the surface $\Delta' := (\phi^*z = 0)$ (resp. for
$\Sigma$) we have $p_*\Delta' \equiv 4\widetilde{\Delta'}$ (resp.
$p_*\Sigma \equiv 4\widetilde{\Sigma}$), where
$\widetilde{\Delta'} := p(\Delta')$ (resp. $\widetilde{\Sigma} :=
p(\Sigma)$). Note that $\Delta' \sim \Delta$ because $\ell
\not\subset (z = 0)$ (cf. {\ref{subsection:pr-11-2}}).

Finally, since both $R_i \in |L|$, from
Lemma~\ref{theorem:ram-div-of-p} we obtain $\widetilde{R_i} :=
p_*R_i \equiv 2\widetilde{L} := p(L)$, and the linear system
$|2\widetilde{L}|$ is basepoint\,-\,free on $X$ (for it is the
$\pi_X^*$ of a free linear system on $\p^2$).

\begin{lemma}
\label{theorem:ac-for-x-0-lcjcj} The pair $(Q, \phi_*\Delta' +
\phi_*L')$ is lc.
\end{lemma}

\begin{proof}
Note that both $(Q, \phi_*\Delta')$ and $(Q,\phi_*L')$ are lc via
lifting these pairs to a small resolution of $Q$. Hence by the
Inversion of Adjunction it suffices to prove that $(\phi_*\Delta',
\phi_*L'\big\vert_{\phi_*\Delta'})$ is lc. But it follows from
{\ref{subsection:pr-11-1}} that $\phi_*\Delta' = (z = 0) \cap Q$
is smooth and the pair $(\phi_*\Delta',
\phi_*L'\big\vert_{\phi_*\Delta'})$ is locally analytically of the
form $(\com^2, (xy = 0))$.
\end{proof}

\begin{lemma}
\label{theorem:ac-for-x-0} $ac(X,D) = \displaystyle\frac{3}{4}$
for an appropriate $D$ (cf. {\ref{subsection:int-3}}).
\end{lemma}

\begin{proof}
In the previous notation, we can write
$$
-K_W = \Sigma + \Delta' + L' + \Delta'',
$$
where $\Delta''\sim \Delta$ is generic (this follows from the fact
that $-K_Q = \phi_*(\Delta' + L' + \Delta'')$ and $K_W = \phi^*K_Q
+ \Sigma$).

Note that the pair $(W,\Sigma + \Delta' + L' + \Delta'')$ is lc,
since
$$
K_W + \Sigma + \Delta' + L' = \phi^*(K_Q + \phi_*\Delta' +
\phi_*L')
$$
by construction, the pair $(Q, \phi_*\Delta' + \phi_*L')$ is lc by
Lemma~\ref{theorem:ac-for-x-0-lcjcj}, $\Delta''$ is generic and
the linear system $|\Delta|$ is basepoint\,-\,free.

Further, from the Hurwitz formula (applied twice) we get
$$
K_W \equiv p^*\big(K_X + \frac{1}{2}\widetilde{R_1} +
\frac{1}{2}\widetilde{R_2} + \frac{1}{2}\widetilde{\Sigma}\big)
$$
(cf. Lemma~\ref{theorem:ram-div-of-p}), which gives
\begin{equation}
\nonumber -2\widetilde{\Sigma}
-4(\widetilde{\Delta'}+\widetilde{L'}) - \widetilde{\Delta''} =
p_*(K_W) \equiv 4(K_X + \frac{1}{2}\widetilde{R_1} +
\frac{1}{2}\widetilde{R_2} + \frac{1}{2}\widetilde{\Sigma}),
\end{equation}
where $\widetilde{\Delta''} := p(\Delta'')$. Then the pair
$$
\big(X,D:=\widetilde{\Sigma}+\widetilde{\Delta'}+\widetilde{L'} +
\frac{1}{4}\widetilde{\Delta''} + \frac{1}{2}\widetilde{R_1} +
\frac{1}{2}\widetilde{R_2}\big)
$$
is also lc because
$$
K_W + \Sigma + \Delta' + L' + \Delta'' \equiv 0 \equiv p^*(K_X +
D).
$$
This gives $ac(X,D) = \displaystyle\frac{3}{4}$ (cf.
Remark~\ref{remark:structure-of-sigma}).\footnote{~Recall that
according to {\ref{subsection:int-2}} we do not distinguish
between $X$ and its $\mathbb{Q}$\,-\,factorialization. Note also
that the divisor $\widetilde{\Sigma}$ is \emph{reducible} and
consists of two non\,-\,$\mathbb{Q}$\,-\,Cartier components.}
\end{proof}

\begin{remark}
\label{remark:no-tor-sl} One can easily modify the pair $(X,D)$
from the proof of Lemma~\ref{theorem:ac-for-x-0} in order to
achieve $ac(X) = 0$ (thus disproving
Conjecture~\ref{theorem:slava} as well). Namely, since
$|2\widetilde{L}|$ is basepoint\,-\,free, we may replace
$\displaystyle\frac{1}{2}\widetilde{R_1} +
\displaystyle\frac{1}{2}\widetilde{R_2}$ by $\widetilde{R_1}$,
say. Then it follows from \cite[Proposition 4.3.2]{pro-comp}
(applied to $\alpha: X \map \p(1,1,1,2)$ as in
{\ref{subsection:pr-11-2}}) that the modified pair $(X,D)$ is
$1$\,-\,complementary. Alternatively, one may notice that
$\alpha_*\widetilde{L'} \equiv \mathcal{O}(1) \equiv
\alpha_*\widetilde{\Delta'}$ for the generator $\mathcal{O}(1)$ of
the class group of $\p(1,1,1,2)$, which easily yields an integral
$D$ with $ac(X,D) = 0$.
\end{remark}

Theorem~\ref{theorem:main-1} (in the $n=3$ case) now follows from
the next

\begin{prop}
\label{theorem:non-toric} Threefold $X$ is
non\,-\,toric.\footnote{~Compare with \cite[Lemma 3.4]{pro-2}.}
\end{prop}

\begin{proof}
Assume the contrary. Recall that $\Sigma = \mathbb{F}_1 \cup
\p^2$. Note also that both ramification divisors $R_i \sim L$
contain the $(-1)$\,-\,curve $e \subset \mathbb{F}_1$ (this curve
lies in a fiber of $\pi$ and the claim follows by descending
everything to $\p^2 = \pi(W)$). In particular, $e$ is acted
trivially by $G$, which implies that $X$ has singularities of the
form $\com^*\times\com^2\slash G$ at the general point of $p(e)$.

Further, if $h \subset \mathbb{F}_1$ is the tautological section,
then we get two more singularities $A_i := R_i \cap \Delta' \cap
h$ of type $\com^*\times\com^2\slash G$, \emph{not} contained in
$p(e)$.

Let $N\simeq\cel^3$ and $\Lambda\subset N\otimes_{\cel}\re$ be the
lattice and the fan, respectively, associated with $X$. Set also
$M := \text{Hom}(N,\cel)$. Fix a cone $\sigma\subset\Lambda$, with
associated affine toric neighborhood
$U_{\sigma}:=\spe~\com\left[\sigma^{\vee}\cap M\right]$ on $X$, so
that $p(e)$ is identified with a $2$\,-\,stratum on $\sigma$. Let
also $\alpha:X\map\p(1,1,1,2)=:\p$ be as at the end of
{\ref{subsection:pr-11-2}}.

By assumption, morphism $\alpha$ is toric,
$\alpha(\widetilde{\Sigma})\subset\p$ is a ``\,line\,'' and
$\Lambda$ looks like as on Figure~\ref{fig-1}.
\begin{figure}
\includegraphics[scale=1.2]{tes.1}
\caption{} \label{fig-1}
\end{figure}
Here $\text{Supp}\,\Lambda$ coincides with the support of the fan
$\Lambda_{\p}$ of $\p$ (i.\,e. $\Lambda$ subdivides
$\Lambda_{\p}$). Note also that vector $OC$ corresponds to the
surface $p(\mathbb{F}_1)$ and the face of $\sigma$ containing $OC$
corresponds to $p(e)$.

Now, since $A_1,A_2\not\in p(e)$, we must also have
$A_1,A_2\not\in U_{\sigma}$. But the latter is impossible for $\p$
containing only one singular point.
Proposition~\ref{theorem:non-toric} is proved.
\end{proof}

In order to prove Theorem~\ref{theorem:main-1} for any $n\ge 4$ it
suffices to take $\frak{X} := X \times (\p^1)^{n-3}$. From
Proposition~\ref{theorem:non-toric} one easily derives that
$\frak{X}$ is non\,-\,toric. At the same time, we have
$ac(\frak{X},\frak{D}) = \displaystyle\frac{3}{4}$, where
$\frak{D}$ is the pullback to $\frak{X}$ of the boundary $D$ on
$X$ with $ac(X,D) = \displaystyle\frac{3}{4}$ (cf.
Lemma~\ref{theorem:ac-for-x-0}), plus the sum of general divisors
= pullbacks w.\,r.\,t. the projection $\frak{X}\map (\p^1)^{n-3}$,
so that $K_{\frak{X}}+\frak{D}\equiv 0$.

\refstepcounter{equation}
\subsection{}
\label{subsection:e-5as}

Fix some $X\in\frak{T}^{f,n}$ (cf.
Definition~\ref{theorem:f-toric}). We conclude by asking the
following questions:

\renewcommand{\theenumi}{\Alph{enumi})}

\begin{enumerate}
\item\label{1-ita} Can $X$ always be obtained as (an extremal contraction$\slash$a blowup of) the quotient
$\p(\ve)/G$ for some toric variety $T$, (semistable) vector bundle
$\ve$ on $T$ and a finite group $G\circlearrowright\p(\ve)$?
\smallskip
\item Does $X$ admit a regular $(\com^*)^k$\,-\,action for some $k\ge
1$ (cf. \cite{prok-del-pez})?
\smallskip
\item Is $X$ a compactification of the torus $(\com^*)^n$ (cf. {\ref{subsection:int-1}} and Remark~\ref{remark:formal-rel})?
\smallskip
\item Is $X$ a Mori dream space$\slash$an FT variety$\slash\ldots$ (see \cite{hu-keel}, \cite{takagi-et-al})?
\end{enumerate}

\bigskip

\renewcommand{\thesubsection}{\bf A.\arabic{equation}}

\renewcommand{\theequation}{A.\arabic{equation}}

\renewcommand{\thesection}{}

\section{}

\section*{Appendix}

Below are a few comments on the paper \cite{brown} (published
recently in Duke Math. J.) In {\ref{subsection:comsds-1}},
{\ref{subsection:comsds-2}}, we show that the proof of
Conjecture~\ref{theorem:slava} given in that paper is
inconsistent.

\refstepcounter{equation}
\subsection{}
\label{subsection:comsds-1}

The most crucial error in \cite{brown} (destroying the whole
argument basically) is at the end of Section 3 on page 25 (two
last paragraphs).

Namely, the authors consider a commutative ring $R$, which is of
dimension $d$ and is acted (faithfully) by the torus
$(\mathbb{C}^*)^r$, some $d
> r$. They claim that once there are $d$ $(\mathbb{C}^*)^r$\,-\,invariant
divisors $T_i$ on $Y = \text{Spec}(R)$ such that the pair $(Y,T_1
+ ... + T_d)$ is log canonical then $R$ must be the polynomial
ring in $d$ variables.

This is totally wrong however.

{\bf N.B.} {\it In fact the authors consider some $R =
\text{Cox}(X)$ and assume silently that $R$ is {\rm positively
graded}. The latter is never proved in the paper (the claim itself
is wrong actually). This, at least, reveals a huge gap in their
argument.}

Let $\mathbb{C}^*$ act on $(x,y,z)$ as follows:
$$(x,y,z) \to (a^{-3}x,ay,z)$$
for all $a \in \mathbb{C}^*$. Now consider the
$\mathbb{C}^*$-invariant surface $Y \subset \mathbb{C}^3$, given
by the equation
$$
xy^3 = z^2 + 1,
$$
and two divisors $T_1 = (x = z - \sqrt{-1} = 0), T_2 = (y = z -
\sqrt{-1} = 0)$ on $Y$. Then all the needed conditions as above
are met. At the same time, we have $Y \ne \mathbb{C}^2$, since the
surface $Y$ has Makar-Limanov invariant equal $\mathbb{C}[y]$ (cf.
\cite[Example 3]{arzh}).

\refstepcounter{equation}
\subsection{}
\label{subsection:comsds-2}

Another (less crucial but still) illustration is the ``\,proof\,"
of Proposition 7.2 in the text.

There the authors appeal to rationality criterion for \emph{any}
(\emph{singular}) conic bundle $3$\,-\,fold via Prym.

{\bf N.B.} {\it Note that their conic bundle $X$ is indeed
singular because $X = Y/(\mathbb{Z}/2)$ for some smooth
$3$\,-\,fold $Y$ and $\mathbb{Z}/2$ acting on it with some $\dim =
1$ fixed locus.}

This criterion is only known for \emph{standard} (in particular
\emph{smooth}) conic bundles and one wonders where did the authors
take from the same fact in the general case. OK, one can consider
\emph{$\mathbb{Q}$\,-\,conic bundles}, and then reduce everything
to the standard case (Sarkisov). But this requires at most
\emph{terminal} singularities --- which is not the case for the
given $X$. (In fact the conclusion of Proposition 7.2 is false ---
see Theorem~\ref{theorem:main} above.)

\bigskip

\bigskip


\end{document}